\documentclass[preprint, 12 pt]{elsarticle}
\usepackage[latin1]{inputenc}
\usepackage{amsmath}
\usepackage{amsthm}
\usepackage{amssymb}
\usepackage{graphicx}
\usepackage{color}
\usepackage{graphicx}% http://ctan.org/pkg/graphicx
\usepackage{multirow}
\usepackage{mathrsfs}
\usepackage{xcolor}
\usepackage{pagecolor}
\theoremstyle{plain}
\newtheorem{thm}{Theorem}[section]

\newtheorem{conj}[thm]{Conjecture}

\newtheorem{defn}[thm]{Definition}

\numberwithin{equation}{section}
\makeatletter
\renewcommand{\subsection}{\@startsection
{subsection}{2}{0mm}{\baselineskip}{-0.25cm}
{\normalfont\normalsize\bf}}
\makeatother

\newcommand{\RNum}[1]{\uppercase\expandafter{\romannumeral #1\relax}}

\journal{---}
\begin{document}
\begin{frontmatter}
\title{A construction for a counterexample to \\ the pseudo $2$-factor isomorphic graph conjecture}

\author{Mari\'en Abreu}	
\ead{marien.abreu@unibas.it}
\author{Martin Funk}
\ead{martin.funk@unibas.it}
\author{Domenico Labbate}
\ead{domenico.labbate@unibas.it}
\author{Federico Romaniello}
\ead{federico.romaniello@unibas.it}

\address{Universit\`{a} degli Studi della Basilicata, Viale dell'Ateneo Lucano 10, 85100 Potenza, Italy.}

%\maketitle
\begin{abstract}

A graph $G$ admiting a $2$-factor is \textit{pseudo $2$-factor isomorphic} if the parity of the number of cycles in all its $2$-factors is the same.
In \cite{ADJLS} some of the authors of this note gave a partial characterisation of pseudo $2$-factor isomorphic bipartite cubic graphs and conjectured that
$K_{3,3}$, the Heawood graph and the Pappus graph are the only essentially $4$-edge-connected ones. In \cite{JG} Jan Goedgebeur computationally found a graph $\mathscr{G}$ on $30$ vertices which is pseudo $2$-factor isomorphic cubic and bipartite, essentially $4$-edge-connected and cyclically $6$-edge-connected, thus refuting the above conjecture. In this note, we describe how such a graph can be constructed from the Heawood graph and the generalised Petersen graph $GP(8,3)$, which are the Levi graphs of the Fano $7_3$ configuration and the M\"obius-Kantor $8_3$ configuration, respectively. Such a description of $\mathscr{G}$ allows us to understand its automorphism group, which has order $144$, using both a geometrical and a graph theoretical approach simultaneously. Moreover we illustrate the uniqueness of this graph.

%J. Goedgebeur computationally found a graph of order $30$ and with automorphism group of order $144$ that disproved a conjectured by Abreu et al.  In this note, we explain how this {\em Goedgebeur's graph} is generated  {\bf cambiare pure vorrei dire come nasce o da dove proviene}.and why it is unique in its kind, furnishing a construction. Particularly, we analyze the structure of  $\Gamma$ describing its generation as a some kind of {\em amalgam}  of the Heawood graph and the Levi graph of the M\"{o}bius-Kantor $8_3$ configuration. This construction will completely explain the automorphism group of \textit{Goedgebeur's graph} as a semidirect product of groups of order $9$ and $16$ respectively.\\

\end{abstract}

\end{frontmatter}
\section{Introduction}

All graphs considered in this note are simple (without loops or multiple edges) and undirected. Most of our terminologies are standard; for further definitions and notation not explicitly stated in the paper, please refer to \cite{BM} and \cite{Wie}.

%We denote by $V(G)$ the vertex set of a graph $G$ and its edge set by $E(G)$.

A $2$-factor of a graph $G$ is a $2$-regular spanning subgraph of $G$. A graph $G$ admiting a $2$-factor is \textit{pseudo $2$-factor isomorphic} if the parity of the number of cycles in all its $2$-factors is the same. In a cubic graph, the three edges incident with a vertex constitute a $3$-edge-cut because their removal leaves an isolated vertex, and is called \emph{trivial}, others being \emph{non-trivial}.
A cubic graph is said to be \emph{essentially $4$-edge-connected} if it contains no non-trivial $3$-edge-cut.
A set $S$ of edges of a graph $G$ is a \emph{cyclic edge cut} if $G-S$ has two components each of which contains
a cycle. We say that a graph $G$ is \emph{cyclically $m$-edge-connected} if each cyclic edge cut of $G$ has size at least
$m$. In \cite{ADJLS} some of the authors of this note gave a partial characterization of pseudo $2$-factor isomorphic bipartite cubic graphs and conjectured:

\begin{conj}\cite[Conjecture $3.6$]{ADJLS}\label{cong}
Let $G$ be an essentially $4$-edge-connected pseudo $2$-factor isomorphic cubic bipartite graph. Then $G$ must be $K_{3,3}$, the Heawood graph or the Pappus graph.
\end{conj}

In \cite{JG} Jan Goedgebeur computationally found a graph $\mathscr{G}$ on $30$ vertices which is pseudo $2$-factor isomorphic cubic and bipartite, essentially $4$-edge-connected and cyclically $6$-edge-connected, thus refuting the above Conjecture \ref{cong} (cf. Figure \ref{Disegno}). Here, we explain how $\mathscr{G}$ (which we will also refer to as \emph{Goedgebeur's graph}) is generated and why it is unique in its kind, providing a construction. In particular, we analyse the structure of $\mathscr{G}$ describing how it arises from the Heawood graph and the generalised Petersen graph $GP(8,3)$, which are the Levi graphs of the Fano $7_3$ configuration $F$ and the M\"obius-Kantor $8_3$ configuration $MK$, respectively. This construction will completely explain the automorphism group of $\mathscr{G}$ as a semidirect product of groups of order $9$ and $16$ respectively. This construction does not generalize in a natural way to an infinite family of graphs preserving all properties of $\mathscr{G}$, in particular being pseudo $2$-factor isomorphic.

Recall that a cubic bipartite graph on $2n$ vertices with girth at least $6$ is the Levi graphs of a symmetric configurations $n_3$.
When counting symmetric configurations, we must distinguish between self-dual configurations, which give rise to one bipartite graph of girth at least $6$ (via its Levi graph), and pairs of distinct configurations dual to each other with isomorphic Levi graphs. With increasing order $n$ the number of pairwise non-isomorphic structures exponentially increases but almost all of them are rigid, i.e. they have a trivial automorphism group. Therefore it makes sense to focus on graphs which have a \textit{large} automorphism group, say at least $4n$.

The graph $\mathscr{G}$ has $30=2n$ vertices and is the Levi graph of a $n_3=15_3$ configuration $\mathscr{C}$. For $n=15$, there are $125571$ cubic bipartite graphs with girth at least $6$, which give rise to $5802$ $15_3$ self-dual configurations and $119770$ pairs of non-isomorphic $15_3$ configurations dual to each other. Note that the automorphism group of the Levi graph of a configuration $n_3$ is twice as big as the one of the corresponding configuration. There are $6$ graphs with \textit{large} automorphism group, namely the Levi graph of the generalised quadrangle of order $2$, with automorphism group of order $2 \cdot 8094$, moreover there are two Levi graphs of self-dual configurations with automorphism group $2 \cdot 720$ and $2 \cdot 128$, respectively, a pair of non-isomorphic configurations with automorphism group of order $192$, as well as Goedgebeur's graph  $\mathscr{G}$, i.e. a self-dual graph with automorphism group of order $144$. All Levi graphs of other $15_3$ configurations have automorphism group of order less than $4n$ (cf. \cite{BPriv}, \cite{BBP}, \cite{JG}).

In Section \ref{const} we will describe the construction that gives rise to $\mathscr{G}$; in Section \ref{auto} we investigate the automorphism group of $\mathscr{G}$ and in Section \ref{unique} we analyse the uniqueness of $\mathscr{G}$.

%%Let $G$, $G_1$, $G_2$ be cubic graphs such that $G_1 \cap G_2 = \emptyset$. Let $y \in V(G_1)$ and $x \in V(G_2)$. Let $x_1$, $x_2$, $x_3$ be the neighbours of $y$ in $G_1$ and $y_1$, $y_2$, $y_3$ be the neighbours of $x$ in $G_2$. If $G=(G_1 - y) \cup (G_2 -x) \cup \left\{x_1y_1,x_2y_2,x_3y_3\right\}$, then we say that $G$ is a \textit{star product} of $G_1$ and $G_2$ and write $G= G_1 * G_2$.\\
%

\section{The Construction}\label{const}
In this section we describe how the configuration $\mathscr{C}$, of which $\mathscr{G}$ is the Levi graph, arises by appropriately joining the Fano configuration $F$ and the M\"obius-Kantor configuration $MK$.
%In terms of configurations, the {\em synthesis} starts between a copy of the Fano Plane $7_3$ on the one hand and a copy of the M\"obius-Kantor $8_3$ configuration on the other hand. To this end, we will give a slightly different meaning to the word \textit{four-gon}:
The following definition will be very useful for our purposes:
\begin{defn}
A quadrilateral consists of $4$  points $P_0,P_1,P_2,P_3$ in general position (i.e. no three collinear) and the $4$ (oriented) lines $P_iP_{i+1}$, indices taken modulo $4$.
%Dually, it also consists of $4$ lines $l_0,l_1,l_2,l_3$ in general position and the (cyclically ordered) points $l_i \cap l_{i+1}$, indices taken modulo $4$.
%A four-gon consists of $4$ points $P_0,P_1,P_2,P_3$ in general position and the $4$ (oriented) sides $P_iP_{i+1}$, indices taken modulo $4$. Dually, a four-side consists of $4$ lines $l_0,l_1,l_2,l_3$ in general position and the (cyclically ordered) points $l_i \cap l_{i+1}$, indices taken modulo $4$.
\end{defn}
Now we define two structures that arise from the configurations $MK$ and $F$ by removing some point/line incidences without deleting neither points nor lines, and changing only the valency of some of them.\\
Consider the classical representation of $MK$ as two quadrilaterals simultaneously inscribed and circumscribed (e.g. cf. \cite[pag. 430]{C}). Disregarding the circumscription, i.e. removing the corrisponding incidences it defines, we obtain an \emph{$MK$-residue} in which the valency of 4 points and 4 lines decreases from three to two. Similarly, removing the incidences of a quadrilateral in $F$, we obtain an \emph{$F$-residue} with 4 points and 4 lines of valency two.\\
The configuration $\mathscr{C}$ then arises by suitably adding incidences among points and lines of valency two between an $F$-residue and an $MK$-residue.

Consider the labellings of the M\"obius-Kantor and Fano plane configurations as in Figure \ref{MK-F} and choose the $MK$-residue and the $F$-residue accordingly.
\begin{figure}[h]
\begin{center}
\includegraphics[scale=0.15]{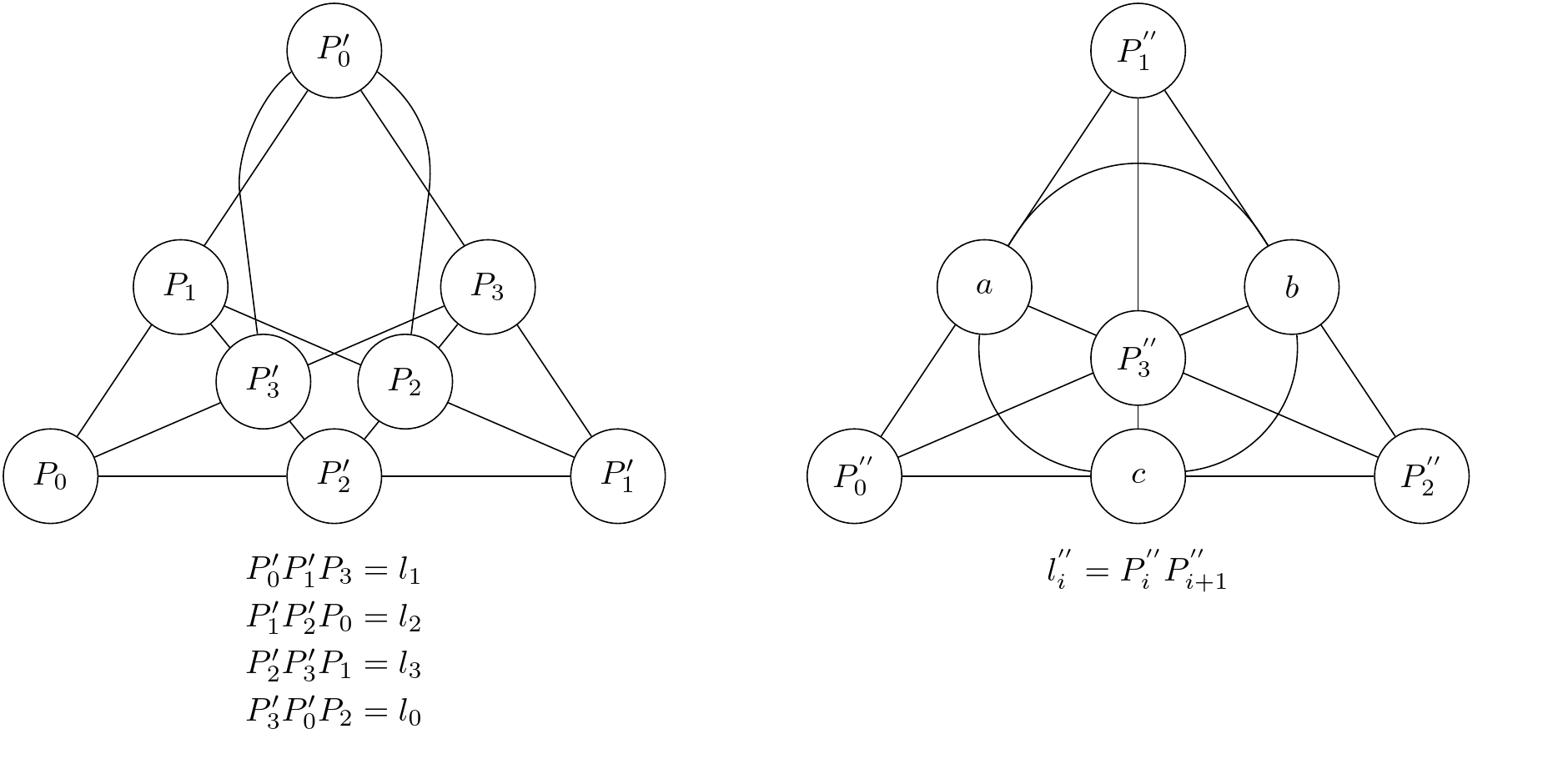}	
\end{center}
\caption{M\"obius-Kantor and Fano plane configurations}	
\label{MK-F}
\end{figure}

Note that the quadrilaterals $P_0P_1P_2P_3$ and $P'_0P'_1P'_2P'_3$ are both diagonal free, in the sense that the pairs $P_0P_2,P_1P_3,P'_0P'_2$ and $P'_1P'_3$ are all non collinear. These two quadrilaterals are in fact mutually inscribed and circumscribed. The $MK$-residue is obtained by removing the circumscription, i.e. the incidences $P_0 | l_2, P_1 | l_3, P_2 | l_0, P_3 | l_1$, where $l_i$ is defined as in Figure \ref{MK-F}. The $F$-residue is obtained by removing the incidences $P''_{i+1} | l''_i$ of the quadrilateral $P''_0,P''_1,P''_2,P''_3$, indices taken modulo $4$, where $l''_i$ is also defined as in Figure \ref{MK-F}.
There are $576$ ways to join the points and lines of the  $MK$-residue and the $F$-residue in order to obtain a $15_3$ configuration (cf. Section \ref{unique}), but the choice of the incidences $P_i | l''_i$ and $P''_i | l_i$, indices taken modulo $4$, gives rise to $\mathscr{C}$ and has $\mathscr{G}$ as its Levi Graph (cf. Figure \ref{Disegno}). In Section \ref{unique} we will analyse what the other choices of incidences give, but first we study $Aut(\mathscr{G})$.

%Thoroughly analysing the structure of $MK$, we realise that there are four pairs of non collinear points, say $\left\{1,2\right\},~\left\{3,4\right\},~\left\{5,6\right\},~\left\{7,8\right\}$. Fixed this labelling there are exactly six diagonal-free quadrilaterals, whereas all other quadrilaterals have exactly one diagonal. Chosen a diagonal-free quadrilateral there is exactly a second one which, with respect to the first, is simultaneously inscribed and circumscribed, namely: $1324 - 5768,~1526 - 3748,~1728 - 3546$. However, the choice of any such  diagonal-free quadrilateral gives rise to a unique $MK$-residue up to isomorphisms.\\
%In $F$ there are 168 ways to select four points in general position. Chosen any of these, since the four points are joined by six lines, there are, up to orientation, three possible quadrilaterals. However, these are all equivalent.
%
\begin{figure}[h]
\begin{center}
\includegraphics[scale=0.15]{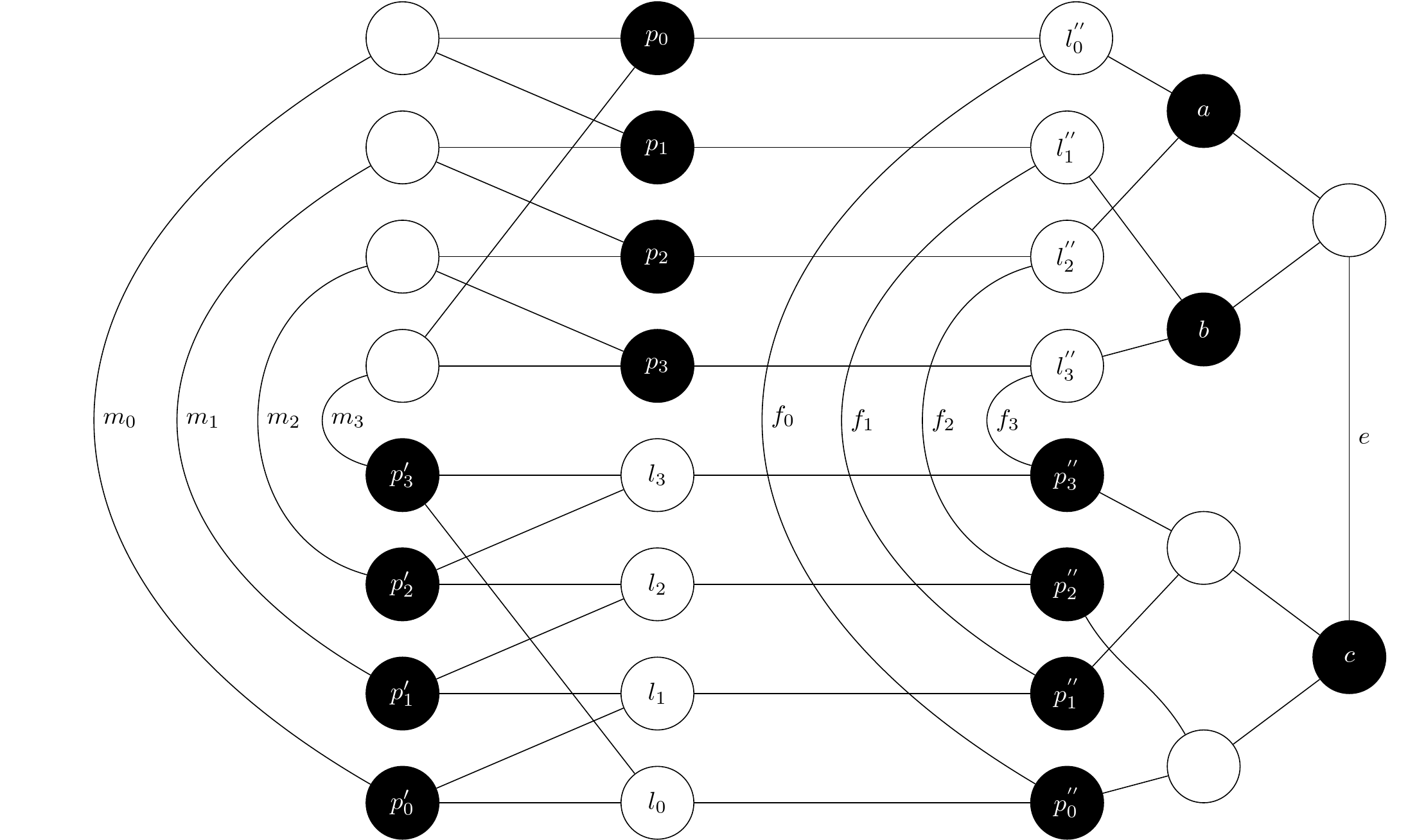}	
\end{center}
\caption{Goedgebeur's graph $\mathscr{G}$}	
\label{Disegno}
\end{figure}
\section{The Automorphism Group of $\mathscr{G}$}\label{auto}
In order to investigate the automorphism group $Aut(\mathscr{G})$, we consider the Levi graph of the $F$-residue and the $MK$-residue described in the previous section and mantain the inheritatd labelling on the vertices, some corresponding to points and other to lines, on different sides of the bipartition.
We refer to the edges connecting the $F$-residue with the $MK$-residue as the \textit{eight-bridge}.
Observe that in $\mathscr{G}$ there is a unique edge $e$ at distance $2$ from each vertex $P''_i, l''_i, i=0, \ldots, 3$ and hence from the eight-bridge. Morover, the edges $f_i=P''_i l''_i, \quad i=0, \ldots, 3$ are independent. Also, in the Levi graph of the MK-residue part of $\mathscr{G}$, there are exactly four independent edges  $m_0, \ldots, m_3$, at distance two from the eight-bridge (cf. Figure \ref{Disegno}).

%Then we have exactly one edge, say $e$, in the $F$-residue at distance two from the eight-bridge, whereas in the $MK$-residue we have four pairwise non-incident edges at distance two, say $m_0,\dots,m_3$. Moreover, in the F-residue we have four pairwise non-incident edges at distance two from $e$, say $f_0,\dots,f_3$. These edges will play an important role and are explicitly labelled in Figure \ref{Disegno}.\\

%%To study $Aut(\Gamma)$ we consider the incidences added between the $F$-residue and the $MK$-residue to obtain $\mathscr{C}$ form \emph{the eight-bridge} in $\Gamma$. Observe that in $\Gamma$ there is a unique edge $e$ at distance $2$ from each vertex in
%%$\cal{P}' \cup \cal{L}'$ and hence from the eight-bridge. Morover, there are exactly four independent edges $f_i=P_i'l_i', \quad i=0, \cdots, 3$. Also, in the Levi graph of the MK-residue part of $\Gamma$, there are exactly four independent edges, say $m_0, \cdots, m_3$ at distance two from the eight-bridge (cf. Figure \ref{Disegno}).\\

\begin{thm}
$Aut(\mathscr{G}) \cong (\mathbb{Z}_3 \times \mathbb{Z}_3) \rtimes (D_4 \times \mathbb{Z}_2)$
\end{thm}

\begin{proof}
 Using Magma Computational Algebra System \cite{BCP}, we computed $Aut(\mathscr{G})$ and realized that it can be described in terms of its action on the set of edges $M:=\left\{e,f_0,\dots,f_3,m_0,\dots,m_3\right\}$, so abusing notation, we will denote automorphisms of $\mathscr{G}$ as permutations of the edges in $M$.

 There are 8 automorphisms characterised by the fact that they map the edge $e$ onto one edge out of $\left\{f_0,\dots,f_3,m_0,\dots,m_3\right\}$. Together with the identity, these automorphisms are a subgroup $K$ of $Aut(\mathscr{G})$ generated by automorphisms $\sigma_0$ and $\sigma_1$:

\begin{center}
\begin{tabular}{c}
$\sigma_0=(e,m_0,m_2)(m_1,f_3,f_0)(m_3,f_2,f_1)$ \\
$ \sigma_1=(e,f_1,f_3)(f_0,m_0,m_3)(f_2,m_1,m_2)$ \\
\end{tabular}
\end{center}

 These automorphisms have order 3, and $K$, which is not cyclic, is hence isomorphic to $\mathbb{Z}_3 \times \mathbb{Z}_3$. Moreover, $K \triangleleft Aut(\mathscr{G})$.

 On the other hand, the stabilizers in $Aut(\mathscr{G})$ of each edge in $M$ represent isomorphic non-abelian subgroups of order 16, none of which is normal in $Aut(\mathscr{G})$. Besides the identity, in each group there are 11 and 4 elements of order 2 and 4 respectively. Using the well-known classification of groups (cf. \cite{Wild}), these are isomorphic to $D_4 \times \mathbb{Z}_2$, and they form a unique class of conjugacy in $Aut(\mathscr{G})$. Hence, we may choose the stabilizer of the edge $e$, say $H$. In $H$, we find a transpostion $\tau$ which inverts the endvertices of the edges in $M$, acting as a horizontal mirror through the center of Figure \ref{Disegno}. Moreover there are two automorphisms:

\begin{center}
\begin{tabular}{c}
$  \delta=(e)(f_0,f_2)(f_1)(f_3)(m_0,m_1)(m_2,m_3)$ \\
$ \rho=(e)(f_0,f_1,f_2,f_3)(m_0,m_1,m_2,m_3)$ \\
\end{tabular}
\end{center}

such that $<\delta, \rho> = D_4$.

In conclusion we have that
$$Aut(\mathscr{G}) \cong K \rtimes H \cong (\mathbb{Z}_3 \times \mathbb{Z}_3) \rtimes (D_4 \times \mathbb{Z}_2).$$
\end{proof}

\section{The uniqueness of $\mathscr{G}$}\label{unique}

In the construction of $\mathscr{G}$, we made a very specific choice of the eight-bridge. We now analyse what the other choices are, and to some extent, what graphs they give rise to. On the first place, we observe that not every eight-bridge preserves bipartition and girth $6$ and that those that do, correspond to two permutations of four elements, i.e. from the symmetric group $S_4$, say $\alpha$ and $\beta$ which are used to establish the edges $P_i | l''_{\alpha(i)}$ and $l_i|P''_{\beta(i)}$. There are $24$ choices for each of $\alpha$ and $\beta$ for a total of $576$ choices for the eight-bridge (preserving bipartition and girth). Let $\mathscr{G}_{\alpha,\beta}$ be the graph with the choice of the eight-bridge according to $\alpha$ and $\beta$. To describe and analyse the choices we consider a partition of the symmetric group $S_4$ as in Table \ref{tabS4} in which the first row contains the normal Klein subgroup $\mathbb{Z}_2 \times \mathbb{Z}_2$, and the other rows its laterals. In the last column of Table \ref{tabS4}, we have represented the effect of, say $\alpha$, on the other adjacencies between $l''_i$ and the vertices $\{a,b\}$ (as in Figure \ref{Disegno}).

\begin{table}[h!]
\begin{center}
\begin{tabular}{||l||c|c|c|c|l|}
\hline
\textbf{Class \RNum{1}}   & $id$       & $(02)$   & $(13)$   & $(02)(13)$ & \begin{minipage}{1.9cm}
      \includegraphics[width=2cm, height= 1.8cm]{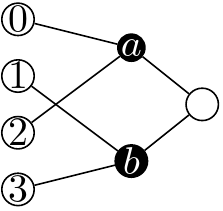}
    \end{minipage} \\ \hline
\textbf{Class \RNum{2}}  & $(01)(23)$ & $(0123)$ & $(0321)$ & $(03)(12)$ & \begin{minipage}{1.9cm}
      \includegraphics[width=2cm, height= 1.8cm]{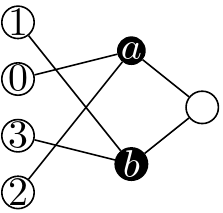}
    \end{minipage} \\ \hline
\textbf{Class \RNum{3}} & $(01)$     & $(012)$  & $(031)$  & $(0312)$ & \begin{minipage}{1.9cm}
      \includegraphics[width=2cm, height= 1.8cm]{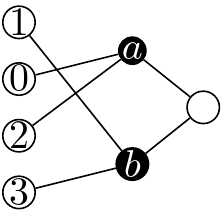}
    \end{minipage} \\ \hline
\textbf{Class \RNum{4}}  & $(23)$     & $(023)$  & $(132)$  & $(0213)$   & \begin{minipage}{1.9cm}
      \includegraphics[width=2cm, height= 1.8cm]{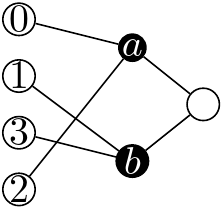}
    \end{minipage} \\ \hline
\textbf{Class \RNum{5}}   & $(03)$     & $(032)$  & $(013)$  & $(0132)$  & \begin{minipage}{1.9cm}
      \includegraphics[width=2cm, height= 1.8cm]{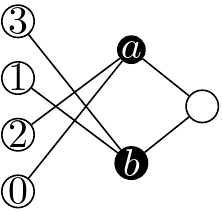}
    \end{minipage} \\ \hline
\textbf{Class \RNum{6}}  & $(12)$     & $(021)$  & $(123)$  & $(0231)$   & \begin{minipage}{1.9cm}
      \includegraphics[width=2cm, height= 1.8cm]{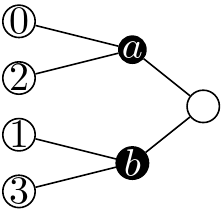}
    \end{minipage} \\ \hline
\end{tabular}
\end{center}
\caption{The partition of $S_4$}\label{tabS4}
\end{table}

\pagebreak

We distinguish among two types of choices of $\alpha$ and $\beta$.

\begin{itemize}
\item \underline{\textbf{Case 1}}: When $\alpha=\beta$
\item \underline{\textbf{Case 2}}: When $\alpha \ne \beta$.
\end{itemize}

In case 1, $\tau$, the mirror automorphism of $\mathscr{G}$ described in the previous section, is still an automorphism of $\mathscr{G}_{\alpha,\beta}$. The choice of $\alpha$ from the first two rows of Table \ref{tabS4} gives rise to graphs isomorphic to $\mathscr{G}$ because the rest of the $F$-residue is preserved. Thus, we have $8$ graphs isomorphic to $\mathscr{G}$. The other $16$ graphs with $\alpha$ from rows $3-6$ of Table \ref{tabS4} do no longer preserve the automorphisms $\sigma_0$ and $\sigma_1$, but they all belong to the same isomorphism class. Such graphs have automorphism group of order $24$, having lost the $\mathbb{Z}_3 \times \mathbb{Z}_3$ part of the automorphism group of $\mathscr{G}$.

In case 2 the situation of symmetries gets much worse, and the results are summarized in Table \ref{tabOther}.

\begin{table}[h!]
\begin{center}
\begin{tabular}{|c|c|c|}
\hline
\textbf{$|Aut(G)|$} & \textbf{$\#$ Isomorphism Classes} & \textbf{$\#$ Representants} \\ \hline
16                  & 1                                 & 8                           \\ \hline
8                   & 6                                 & 16                          \\ \hline
\multirow{2}{*}{4}  & 1                                 & 64                          \\ \cline{2-3}
                    & 4                                 & 32                          \\ \hline
2                   & 2                                 & 64                          \\ \hline
1                   & 1                                 & 128                         \\ \hline
\end{tabular}
\end{center}
\caption{Summary of $\mathscr{G}_{\alpha,\beta}$}\label{tabOther}
\end{table}

In conclusion, the graph $\mathscr{G}$ is unique, up to automorphisms of the $F$-residue, which as described in case 1.

There is also another sense in which $\mathscr{G}$ seems to be unique. As far as we could check, the construction of joining residues of Levi graphs of $n_3$ configurations, do not preserve, in general a strong property such as being pseudo $2$-factor isomorphic. The problem lies in the kinds of cycles which are created that are, thus, able to produce $2$-factors of both parities. Intuitively, there is too much room, whereas the MK-residue and the F-residue are very compact. Moreover, also using several copies of the same residues does not preserve the behavior of the parity of cycles in a $2$--factor.

\section*{Acknowledgement}

The research that led to the present note was partially supported by the group \\ GNSAGA of INdAM.

\end{document}